\documentclass[conference,a4paper]{IEEEtran}
\IEEEoverridecommandlockouts
\usepackage{cite}
\usepackage{amsmath,amssymb,amsfonts}
\usepackage{algorithmic}
\usepackage{algorithm}
\usepackage{graphicx}
\usepackage{float}
\usepackage{hyperref}
\usepackage{subcaption}
\usepackage{tikz}
\usetikzlibrary{positioning,arrows.meta}
\usepackage{textcomp}
\usepackage{circuitikz}
\usetikzlibrary{fit, positioning}
\usepackage[percent]{overpic}  
\usepackage{xcolor}
\newcommand{\blue}[1]{{\color{black}#1}}   
\def\BibTeX{{\rm B\kern-.05em{\sc i\kern-.025em b}\kern-.08em
    T\kern-.1667em\lower.7ex\hbox{E}\kern-.125emX}}

\newtheorem{theorem}{Theorem}

\newtheorem{lemma}{Lemma}

\newtheorem{remark}{Remark}

\newcommand{\X}{\mathbb{X}}
\newcommand{\U}{\mathbb{U}}
\newcommand{\Hx}{\mathcal{H}_x}
\newcommand{\Hu}{\mathcal{H}_u}
\newcommand{\Haug}{\mathcal{H}}
\newcommand{\Psix}{\Psi_x}
\newcommand{\Psiu}{\Psi_u}
\newcommand{\nz}{n_z}
\newcommand{\nv}{n_v}
\newcommand{\nx}{n_x}
\newcommand{\nuu}{n_u}
\newcommand{\Real}{\mathbb{R}}
\newcommand{\Cplx}{\mathbb{C}}
\newcommand{\Nat}{\mathbb{N}}
\DeclareMathOperator{\diag}{diag}
\DeclareMathOperator{\rank}{rank}

\DeclareMathOperator{\spn}{span}

\newcommand{\kron}{\otimes}
\newcommand{\krrao}{\odot}
\newcommand{\Phimat}{\Phi}

\begin{document}

\title{Koopman operator learning for predictive control via  Khatri-Rao kernel regression}

\author{\IEEEauthorblockN{Mircea Lazar}
\IEEEauthorblockA{\textit{Eindhoven University of Technology}\\
Eindhoven, Netherlands \\
m.lazar@tue.nl}
}

\maketitle

\begin{abstract}
This paper develops a data-driven realization of the generalized Koopman operator (GeKo) \cite{Lazar2025}, in which states and inputs are lifted independently and the dynamics are expressed as a tensor bilinear system. The first contribution is a time-sequenced multi-step Khatri--Rao kernel regression formulation that exposes the operator to evolved snapshots along trajectories rather than only single one-step pairs, which reduces compounded prediction error. Secondly, we develop a kernel- and input-agnostic structured SVD reduction that compresses the lifted state and input spaces while preserving the Khatri--Rao realization. We instantiate the framework with random Fourier features and describe a complete predictive-control pipeline, including a multi-step roll-out diagnostic that guides the choice of MPC horizon. The framework is validated on the chaotic Lorenz system, where the learned reduced-order GeKo model \blue{reduces the long-horizon prediction error by up to a factor of five relative to one-step and lifted linear baselines fitted on the same data, and} stabilizes an unstable equilibrium.
\end{abstract}

\begin{IEEEkeywords}
Koopman operator theory for control, kernel regression, singular value decomposition, model predictive control
\end{IEEEkeywords}
\newif\ifl
\lfalse  
 
\section{Introduction}
\label{sec:introduction}
 
The Koopman operator framework~\cite{Koopman1931,
Mauroy2020book} casts a (possibly nonlinear) dynamical system as a
\emph{linear} operator acting on a function space of observables. Over
the past decade, finite-dimensional approximations of the Koopman
operator, e.g., Extended Dynamic Mode Decomposition (EDMD)
\cite{Williams2015} and its variants, have become a standard
data-driven tool for system identification and control of nonlinear
systems~\cite{Mauroy2020book,Strasser2026}.
 
The original Koopman operator, however, applies to autonomous systems.
Extending the framework to controlled systems is non-trivial and has
produced a proliferation of approaches as follows. \emph{Lifted linear models}, e.g.\ Dynamic Mode Decomposition
    with control (DMDc)~\cite{Proctor2016} and 
    linear-in-control Koopman~\cite{KordaMezic2018}, treat the input as an external linear
    forcing of the lifted state. The resulting linear MPC formulation
    is computationally attractive, but neglects input nonlinearities
 and input--state coupling, which degrades
    closed-loop performance for non-input-affine
    systems~\cite{Strasser2026}. \emph{Lifted bilinear models}~\cite{KordaMezic2018, Peitz2020,
    GoswamiPaley2022} introduce a separate lifted
    operator per input channel, which handles control-affine systems, but is challenged by
    strongly nonlinear input functions (e.g., saturations, dead-zones, non-monotonic actuation). 
    \emph{Deep Koopman models}~\cite{Lusch2018} use
    neural-network encoders to learn the lifting, often with
    consistency or contrastive losses. They achieve good empirical
 performance, but lack rigorous convergence theory, and are
    susceptible to local minima during training. Recent works on Koopman MPC \cite{WorthmannCDC2024, DeJongCDC2024, SchimpernaOffset2024, BoldKernelKoopman2025, StrasserSafEDMD2026} provide formal closed--loop guarantees for linear- or bilinear-in-control data-driven Koopman models.
 
A novel theoretical framework that generalizes the Koopman operator to nonlinear systems with inputs was recently proposed in~\cite{Lazar2025}: the \emph{generalized Koopman operator} (GeKo) with a target space defined as the \emph{product Hilbert space} $\Haug = \Hx \otimes \Hu$ formed by the
state observable space $\Hx$ and an independently chosen
input observable space $\Hu$. The GeKo
$\mathcal{K}: \Hx \to \Haug$ acts by composition with the dynamics mapping $F(x,u)$, i.e., $(\mathcal{K} \psi_x)(x,u) = (\psi_x \circ F)(x,u)$, capturing any input
nonlinearity natively, without structural restrictions. In a finite basis it yields
the input--state separable Khatri--Rao realization
$z_{t+1} = K\,(z_t \otimes v_t)$, with lifted state $z_t = \Psix(x_t)$, lifted input $v_t =
\Psiu(u_t)$, and operator matrix $K$. The GeKo realization is provably more general than the lifted linear and bilinear forms (which it recovers as special
cases~\cite{Lazar2025}), opening a promising route to
data-driven predictive control of complex nonlinear systems. While \cite{Lazar2025} introduces the Khatri--Rao kernel regression and illustrates its performance in MPC for a specific example, it does not further investigate how the data-driven learning of GeKo models can be optimized for usage in MPC (which requires low-complexity models with high accuracy).  This is the subject of the present paper.
\begin{remark}
An alternative recent extension of the Koopman operator to control is the \emph{Koopman Control Family} (KCF) approach~\cite{HaseliCortesKCF2026}, defined as the family of Koopman operators associated with all constant-input dynamics. 
Both KCF and GeKo encompass the lifted linear- and bilinear-in-control Koopman forms as special cases. The two points of view are complementary: KCF arrives at a separable Koopman form in the state and input through a family of operators indexed by constant inputs, while GeKo lifts the input directly via $\Psiu$ and produces a unique linear operator on a product Hilbert space, which is separable by construction.
\end{remark}
 
\subsection*{Problem statement and contributions}
 
Given trajectory state--input data from a
nonlinear control system, we seek a finite-dimensional GeKo
realization that is suitable for MPC, i.e., \emph{(i)}~\emph{accurate over multi-step horizons}, not just one-step; \emph{(ii)}~\emph{convergent and well-posed under reduction}, so that the inevitable rank deficiency of large kernel dictionaries can be exploited to compress the model without losing the structure of the GeKo realization.
 
The contributions in this paper address the above goals as follows. After briefly recalling the GeKo framework \cite{Lazar2025} in Section~\ref{sec:geko}, Section~\ref{sec:multistep} develops a multi-step optimized Khatri-Rao kernel regression formulation over \emph{horizontally stacked, time-shifted snapshot pairs} drawn from trajectories, which sharply reduces the compounded prediction error in multi-step roll-outs. Section~\ref{sec:reduction} then derives a structured SVD reduction, applied independently on the state lifting $z = \Psix(x)$ and the input lifting $v = \Psiu(u)$, that compresses the model to the effective rank of the data while preserving the Khatri--Rao GeKo realization.
\ifl We prove that the reduced model remains a valid Khatri--Rao GeKo realization and we indicate how the approximation error can be bounded by the discarded singular values.
\fi
Section~\ref{sec:rff} instantiates the framework with random Fourier features and presents the complete identification and predictive-control pipeline. Section~\ref{sec:experiments} validates the framework on the chaotic Lorenz system, stabilized at an unstable equilibrium from multiple initial states.
\ifl
the chaotic Lorenz system (stabilized at an unstable equilibrium), and the DC motor benchmark of~\cite{HaseliCortes2026}.
\fi
\subsection*{Basic notation}
$\Real, \Cplx, \Nat$ denote the sets of real, complex and non-negative
integer numbers. For a matrix $A \in \Real^{m \times n}$, $A^\top$ is
the transpose, $A^\dagger$ the Moore--Penrose pseudoinverse,
$\|A\|_F$ the Frobenius norm, $\|A\|_2$ the spectral norm, and
$\rank(A)$ the rank. $A \kron B$ denotes the Kronecker product and
$A \krrao B$ the Khatri--Rao (column-wise Kronecker) product.
$\spn(\Psi)$ denotes the linear span of the elements of a
vector-valued function $\Psi$. For a positive-semidefinite matrix
$Q$, $\|x\|^2_Q := x^\top Q x$.
 
 
\section{The Generalized Koopman Operator and Its Khatri--Rao Realization}
\label{sec:geko}
 
In this section we recall the generalized Koopman operator of~\cite{Lazar2025} and its finite-dimensional Khatri--Rao realization that will serve as the prediction model in the data-driven MPC pipeline developed in this paper. We consider discrete-time nonlinear control systems of the form
\begin{equation}
    x_{t+1} = F(x_t, u_t), \quad x_t \in \X \subseteq \Real^{\nx},
                              \; u_t \in \U \subseteq \Real^{\nuu},
\label{eq:sys}
\end{equation}
where $F:\X\times\U \to \X$ is, in general, neither linear in the state nor affine in the input. \blue{The map $F$ is assumed to be \emph{unknown}: no model or structural information about $F$ is available, and only sampled state--input trajectory data generated by~\eqref{eq:sys} can be used for prediction and control design. In particular, the input may enter through arbitrary nonlinearities, e.g., non-monotonic, saturated or sigmoidal actuation characteristics.}

\blue{The idea underlying Koopman operator methods is to trade nonlinearity for dimension: instead of propagating the state $x_t$ through the unknown nonlinear map $F$, one propagates a vector of scalar functions of the state (observables), whose evolution admits a simple, here, bilinear in the state and input features, representation, at the price of a higher-dimensional state/input.}
To lift~\eqref{eq:sys} into an operator-theoretic representation, let $\Hx \subseteq L^2(\X,\mu_x)$ and $\Hu \subseteq L^2(\U,\mu_u)$ be the Hilbert spaces of state and input observables, and choose finite (truncated) bases
\begin{equation}
    \Psix : \X \to \Real^{\nz}, \qquad
    \Psiu : \U \to \Real^{\nv},
\label{eq:bases}
\end{equation}
that span subspaces of $\Hx$ and $\Hu$ respectively. The corresponding \emph{product Hilbert space} is $\Haug = \Hx \otimes \Hu$, equipped with the tensor-product basis $\Psi:=\Psix \otimes \Psiu$. On this space, the \emph{generalized Koopman operator} (GeKo) of~\cite{Lazar2025} is defined for any $\psi_x\in\Hx$ as
\begin{equation}
    \mathcal{K} :  \Hx \to \Haug,
    \quad
    (\mathcal{K} \psi_x)(x,u) = (\psi_x \circ F)(x,u).
\label{eq:geko}
\end{equation}
As shown in \cite{Lazar2025}, defining the lifted state and input as $z_t := \Psix(x_t)$ and $v_t := \Psiu(u_t)$, respectively, the finite-dimensional realization of~\eqref{eq:geko} in the chosen basis takes the input--state separable form
\begin{equation}
    z_{t+1} = K \, \big(z_t \kron v_t\big),
    \quad K \in \Real^{\nz \times \nz \nv},
\label{eq:khatri-rao}
\end{equation}
where $\kron$ is the Kronecker product. We refer to~\eqref{eq:khatri-rao} as the \emph{Khatri--Rao realization} of the GeKo: the matrix $K$ linearly combines the bilinear state--input features $z_t \kron v_t$ to produce the next lifted state.
 
To fit the matrix $K$ from data, suppose we have a snapshot data set $\{(x_i, u_i, x_i^+)\}_{i=1}^M$ with $x_i^+ = F(x_i, u_i)$, and define the data matrices
\begin{align}
\label{eq:data-matrices}
    Z   &= [\Psix(x_1), \ldots, \Psix(x_M)]   \in \Real^{\nz \times M}, \nonumber\\
    Z^+ &= [\Psix(x_1^+), \ldots, \Psix(x_M^+)] \in \Real^{\nz \times M}, \\
    V   &= [\Psiu(u_1), \ldots, \Psiu(u_M)]   \in \Real^{\nv \times M}. \nonumber
\end{align}
The Khatri--Rao feature matrix is the column-wise Kronecker product
\begin{equation}
    \Phimat = Z \krrao V \in \Real^{\nz\nv \times M},
    \quad \Phimat(:,i) = z_i \kron v_i.
\label{eq:Phi}
\end{equation}
The standard one-step EDMD-type estimator of $K$ is then the ridge-regularized least-squares solution
\begin{equation}
    \hat K
    = \arg\min_K \big\|Z^+ - K\Phimat\big\|_F^2 + \gamma \|K\|_F^2
    = Z^+ \Phimat^\top \big(\Phimat \Phimat^\top + \gamma I\big)^{-1}
\label{eq:edmd}
\end{equation}
with ridge parameter $\gamma > 0$.
 
While~\eqref{eq:edmd} provides a closed-form estimate of the GeKo, it is well known~\cite{KordaMezic2018,HaseliCortes2026,Conradie2026} that minimizing the one-step residual does not imply small multi-step prediction error, since the roll-out predictor compounds errors through the bilinear feature update. This compounding is more severe in the Khatri--Rao realization~\eqref{eq:khatri-rao} than in the autonomous Koopman setting due to the bilinearity in $(z,v)$. The next section addresses this issue by developing a multi-step regression that exposes $K$ to evolved snapshot pairs drawn along trajectories.
 

\section{Time-Sequenced Multi-Step Khatri--Rao Kernel Regression}
\label{sec:multistep}

In this section we develop a regression scheme that fits the finite-dimensional operator $K$ of~\eqref{eq:khatri-rao} from \emph{trajectories} rather than from isolated one-step snapshot pairs, by horizontally stacking time-shifted residuals into a single regression problem. We then introduce a multi-step roll-out diagnostic, evaluated both on the lifted state and on the original state, that quantifies the long-horizon prediction quality of the identified model and guides the choice of MPC prediction horizon.

Suppose we have $L$ training trajectories of the system~\eqref{eq:sys},
each generated by exciting the system with a rich (e.g.\ multi-frequency)
input signal,
\begin{equation}
    \mathcal{D}_\ell = \big\{ (x^{(\ell)}_t, u^{(\ell)}_t)
        : t = 0,\ldots,T_\ell\big\}, \quad \ell = 1,\ldots,L.
\label{eq:trajdata}
\end{equation}
Pick a horizon $N_d \in \Nat$ and define, for each starting time
$t \in \{0,\ldots,T_\ell - N_d\}$ in a trajectory $\ell$, the
time-sequenced snapshot window
\begin{equation}
    \mathcal{W}_{\ell,t}
        = \big\{ (z^{(\ell)}_{t+k},\, v^{(\ell)}_{t+k},\,
                  z^{(\ell)}_{t+k+1})
                  : k = 0,\ldots,N_d-1 \big\},
\label{eq:window}
\end{equation}
where $z^{(\ell)}_t = \Psix(x^{(\ell)}_t)$ and
$v^{(\ell)}_t = \Psiu(u^{(\ell)}_t)$ for all $t\geq 0$.
Let
\[
    \mathcal{S} = \big\{ (\ell_i, t_i) : i = 1,\ldots,M_a \big\}
\]
be a representative set of $M_a$ such windows, where $\ell_i$ and
$t_i$ denote the trajectory index and starting time of the $i$-th
selected window. \blue{Selecting a subset is motivated by redundancy and tractability, as consecutive windows overlap in $N_d-1$ samples and are therefore highly redundant, while the regression developed below scales with the number of stacked pairs $N_d M_a$; a well-spread subset of $M_a \ll M$ windows retains the information content of the raw pool at a fraction of the cost.} In practice, $\mathcal{S}$ is constructed by
\emph{window-aware} $k$-means clustering: each window is summarized by
the concatenation of three anchor snapshots
$[z^{(\ell)}_t;\, z^{(\ell)}_{t+\lfloor N_d/2\rfloor};\, z^{(\ell)}_{t+N_d}]$, i.e.,
its start, midpoint, and end, and one window nearest to each cluster
centroid is retained. Using all three anchors rather than the start
$z^{(\ell)}_t$ alone keeps windows that share a starting point but diverge
along the trajectory (a common case for complex or chaotic systems) as distinct
clusters, ensuring broad coverage of the data manifold.
For each time-shift $k \in \{0,1,\ldots,N_d-1\}$, define the
\emph{shifted snapshot blocks}
\begin{equation}
    \Phimat_k \in \Real^{\nz\nv \times M_a}, \quad
    W_k       \in \Real^{\nz \times M_a},
\label{eq:phi-w-blocks}
\end{equation}
whose $i$-th columns are, respectively,
$\Phimat_k(:,i) = z^{(\ell_i)}_{t_i+k} \kron v^{(\ell_i)}_{t_i+k}$
and $W_k(:,i) = z^{(\ell_i)}_{t_i+k+1}$. Stacking these blocks
horizontally yields the \emph{augmented data matrices}
\begin{equation}
    \Phimat := [\Phimat_0,\, \Phimat_1, \ldots, \Phimat_{N_d-1}], \quad
    W := [W_0,\, W_1, \ldots, W_{N_d-1}],
\label{eq:aug-data}
\end{equation}
with $\Phimat \in \Real^{\nz\nv \times M_{\rm aug}}$,
$W \in \Real^{\nz \times M_{\rm aug}}$, and
$M_{\rm aug} := N_d M_a$.

The \emph{time-sequenced multi-step Khatri--Rao kernel regression} estimator is defined by
\begin{equation}
    \hat K := \arg\min_{K \in \Real^{\nz \times \nz\nv}}
        \;\sum_{k=0}^{N_d-1} \big\| W_k - K\Phimat_k \big\|_F^2
        + \gamma \|K\|_F^2.
\label{eq:multistep-ls}
\end{equation}
By the block structure of~\eqref{eq:aug-data}, we show next that problem~\eqref{eq:multistep-ls} is equivalent to the single ridge regression
\begin{equation}
    \hat K = \arg\min_{K \in \Real^{\nz \times \nz\nv}}
        \big\| W - K\,\Phimat \big\|_F^2 + \gamma \|K\|_F^2,
\label{eq:multistep-single}
\end{equation}
with closed-form solution
\begin{equation}
    \hat K
    = W \Phimat^\top \big(\Phimat\Phimat^\top + \gamma I\big)^{-1}.
\label{eq:multistep-soln}
\end{equation}
\begin{lemma}[Frobenius decomposition of block-stacked least squares]
\label{lem:frobenius-decomp}
Let $\Phimat_k, W_k$, $k = 0,\ldots,N_d-1$, be the shifted snapshot
blocks defined in~\eqref{eq:phi-w-blocks}, and let
$\Phimat, W$ denote their horizontal concatenations
as in~\eqref{eq:aug-data}. Then for any matrix
$K \in \Real^{\nz \times \nz\nv}$ and any $\gamma \ge 0$,
\begin{equation}
    \sum_{k=0}^{N_d-1} \|W_k - K \Phimat_k\|_F^2 + \gamma \|K\|_F^2
    \;=\; \|W - K \Phimat\|_F^2 + \gamma \|K\|_F^2.
\label{eq:frob-decomp}
\end{equation}
\end{lemma}
\begin{IEEEproof}
For any matrices $A_0,\ldots,A_{N-1}$ with the same row dimension,
the definition of the Frobenius norm as the sum of squared entries
gives
$\|[A_0,\ldots,A_{N-1}]\|_F^2 = \sum_{k=0}^{N-1} \|A_k\|_F^2$.
Since $K$ multiplies each block independently,
$[W_0 - K\Phimat_0,\ldots,W_{N-1} - K\Phimat_{N-1}] = W - K\Phimat$,
and \eqref{eq:frob-decomp} follows. The ridge term
$\gamma\|K\|_F^2$ is identical in both expressions.
\end{IEEEproof}

For $N_d=1$ the time-sequenced multi-step Khatri-Rao regression collapses to the standard one-step EDMD estimator~\eqref{eq:edmd}; for $N_d>1$, the regression averages the one-step residual across all $N_d$ time-shift positions of each window, exposing $K$ to evolved snapshots drawn from along the trajectories rather than only window-start states. 

\begin{remark}[Multi-step Koopman learning overview] Fitting a Koopman model over multiple time steps has been explored in the lifted \emph{linear-in-control} setting via a multi-step loss for the input matrix of a lifted linear predictor~\cite{KordaMezic2018}, via neural-network observables trained on a multi-step output error with the horizon-stacked prediction matrices subsequently refit by least squares~\cite{DeJongCDC2024}, and via independent convex per-step fits of condensed horizon-dependent state-input mappings~\cite{WuMultistep2026}. \blue{These approaches trade the single recursive operator for a fixed-horizon stacked predictor, which avoids error compounding, but yields multi-step predictors.}
\end{remark}

In contrast with the above approaches the estimator~\eqref{eq:multistep-soln}  fits a \emph{single} matrix $K$ valid across all time shifts, \emph{preserves the closed-form convex ridge solution despite bilinear lifted dynamics} (a major advantage compared to multi-step Koopman deep learning, which is typically non--convex), and remains usable at arbitrary horizons within the recursive prediction using~\eqref{eq:khatri-rao}. 
\ifl
\begin{remark}
A useful complementary interpretation comes from the reduction of variance. Under the standard assumption that the trajectory residuals are approximately uncorrelated across time shifts, augmenting the ridge regression with $N$ horizontally stacked snapshot blocks effectively multiplies the sample size, reducing the variance of $\hat K$ in the directions spanned by the augmented data. We do not formalize this further for brevity, but we report empirically in Section~\ref{sec:experiments} that the multi-step estimator~\eqref{eq:multistep-soln} produces models whose long-horizon prediction error is substantially smaller than that of the one-step counterpart, even at fixed dictionary size.
\end{remark}
\fi

To quantify multi-step prediction quality and to guide the choice of MPC
horizon, we evaluate two complementary relative error profiles over the
set $\mathcal{S}$ of $M_a$ selected windows.
Let $\hat z_k^{(i)}$ denote the $k$-step roll-out of the identified
Khatri--Rao model from $z_0^{(i)} = \Psix(x_0^{(i)})$ using the true input
sequence $\{u_{0:k-1}^{(i)}\}$ of window $i$, let $z_k^{(i)}$ be the
corresponding true lifted state, and let $D$ be the linear decoder
$\hat x = D z$ (constructed in Section~\ref{sec:reduction}). Collecting the
ensemble at step $k$ into matrices
$\hat Z_k = [\,\hat z_k^{(i)}\,]_{i\in\mathcal{S}}$ and
$Z_k = [\,z_k^{(i)}\,]_{i\in\mathcal{S}}$, we define the lifted-state and
decoded-state relative errors as the relative Frobenius-norm errors
\begin{equation}
\begin{split}
    r^{z}_k
    &= \frac{\|\hat Z_k - Z_k\|_F}{\|Z_k\|_F}
    = \!\sqrt{
        \frac{\sum_{i \in \mathcal{S}}
              \|\hat z_k^{(i)} - z_k^{(i)}\|^2}
             {\sum_{i \in \mathcal{S}} \|z_k^{(i)}\|^2}
    },\\
    r^{x}_k
    &= \frac{\|D\hat Z_k - D Z_k\|_F}{\|D Z_k\|_F}
    = \!\sqrt{
        \frac{\sum_{i \in \mathcal{S}}
              \|D\hat z_k^{(i)} - D z_k^{(i)}\|^2}
             {\sum_{i \in \mathcal{S}} \|D z_k^{(i)}\|^2}
    },
\end{split}
\label{eq:multistep-profile}
\end{equation}
for $k = 1,\ldots,N$. We normalize by the signal norm rather than reporting
an absolute error because the states of a dynamical system may differ substantially in magnitude and because the lifted-feature and
decoded-state errors carry different units; the normalization renders both profiles dimensionless and mutually comparable, and lets a single
scale-independent threshold ($1\%$, $5\%$) drive horizon selection.
Aggregating the squared error and the squared signal norm separately over
$\mathcal{S}$ -- rather than averaging per-window ratios -- prevents
low-magnitude windows from dominating the metric. The lifted profile
$r^{z}_k$ measures error growth in the full $\nz$-dimensional feature space
natural to the Koopman operator, while the decoded profile $r^{x}_k$
measures growth in the original state space, relevant to MPC performance.

In our experiments, $r^x_k$ is consistently smaller than $r^z_k$ and degrades more gracefully with $k$\blue{; note that this is a property of the least-squares decoder $D$ identified by~\eqref{eq:decoder-soln} in Section~\ref{sec:reduction}, not of an arbitrary decoding matrix:} many feature directions that drift over time are projected out by \blue{the fitted} $D$ when only the state contributes to the controller objective. We thus use $r^x_k$ as the primary horizon-selection diagnostic and report $r^z_k$ as a secondary check. Given a tolerance $\epsilon \in (0,1)$, the MPC prediction horizon can be chosen based on the threshold
\begin{equation}
    N_\text{mpc} \le \max\big\{\, k \in \{1,\ldots,N_d\} \;:\; r^x_k \le \epsilon\,\big\}.
\label{eq:Hp-rule}
\end{equation}
\blue{We stress that~\eqref{eq:Hp-rule} is a practical design heuristic rather than a guarantee, as a small open-loop roll-out error over the horizon is a prerequisite for, but not a certificate of, closed-loop performance. Future work will research guarantees. }
\ifl
so that the MPC stage cost is evaluated on a regime where the decoded-state prediction is empirically accurate to within $\epsilon$. This rule is conservative -- it relies only on roll-outs along training trajectories -- but it has proven reliable across all benchmarks in Section~\ref{sec:experiments}, including the chaotic Lorenz system, where the safe-horizon predicted by~\eqref{eq:Hp-rule} matches the empirical closed-loop accuracy obtained by MPC.
\fi

\section{Data Generation and\\ Structured SVD Reduction}
\label{sec:reduction}

Next we specify the input excitation used to generate the trajectories~\eqref{eq:trajdata} and then develop a kernel-agnostic structured SVD reduction of the state and input liftings $\Psix$, $\Psiu$ that compresses the model while preserving the GeKo form~\eqref{eq:khatri-rao}.

\subsection*{Trajectory data generation}

We generate the $L$ training trajectories from initial conditions $\{x_0^{(\ell)}\}_{\ell=1}^L$ drawn to cover $\X$ (e.g.,\ uniformly on a bounding box, or by clustering of a pilot run). Each trajectory is excited with a multi-frequency input
\begin{equation}
    u^{(\ell)}_k \;=\; \mathrm{sat}_{\U}\!\left(
        \sum_{j=1}^{n_f}
        \alpha_j \sin\big(\omega_j\, k T_s + \varphi_j^{(\ell)}\big)
    \right),
\label{eq:input-design}
\end{equation}
with $k \in \mathbb{N}$ the sample index and $T_s$ the sampling time (so that $k T_s$ is the elapsed time), frequencies $\{\omega_j\}_{j=1}^{n_f}$ chosen to span the dynamic system bandwidth, amplitudes $\alpha_j$ tuned so that the input saturates only occasionally, phases $\varphi_j^{(\ell)}$ randomized per trajectory, and $\mathrm{sat}_\U(\cdot)$ the saturation enforcing $u \in \U$.
\ifl Compared with piecewise-constant random inputs, this multi-frequency design provides broad spectral content without abrupt transients that degrade the discrete-time approximation.\fi

\subsection*{Structured SVD reduction}

Random or kernel-based dictionaries are intentionally over-parameterized: with random Fourier features (Section~\ref{sec:rff}), the recommended dictionary size $\nz$ scales with the desired kernel approximation quality~\cite{RahimiRecht2007} and is typically much larger than the effective dimension of the data manifold. As a consequence, the lifted state matrix $Z$ has numerical rank far below $\nz$, and the Khatri--Rao feature matrix $\Phimat = Z \krrao V$ inherits and amplifies this rank deficiency. We therefore reduce both liftings to their effective ranks via SVD, independently in $z$ and $v$.

For the state lifting, let $\bar Z := [Z\;\;Z^+] \in \Real^{\nz \times 2M}$ collect all lifted state samples (current and next), and consider the eigendecomposition of the symmetric positive semidefinite Gram matrix
\begin{equation}
    G_z \;:=\; \bar Z \bar Z^\top
        \;=\; U_z\, \Sigma_z^2\, U_z^\top,
\label{eq:Gz}
\end{equation}
with $U_z \in \Real^{\nz \times \nz}$ orthonormal and $\Sigma_z = \diag(\sigma_1^z, \ldots, \sigma_{\nz}^z)$ sorted in decreasing order. Choose the truncation rank $r_z \le \nz$ as the smallest index such that $\sigma_{r_z+1}^z \le \tau_z \, \sigma_1^z$ for a relative tolerance $\tau_z > 0$, and let $U_{z,r_z} := U_z(:,1{:}r_z)$. The \emph{reduced state lifting} is then
\begin{equation}
    \tilde \Psi_x(x) \;:=\; U_{z,r_z}^\top\, \Psix(x) \;\in\; \Real^{r_z},
\label{eq:reduce-z}
\end{equation}
and we write $\tilde z := U_{z,r_z}^\top z$. The same construction applied to the input data Gram matrix $G_v := V V^\top = U_v \Sigma_v^2 U_v^\top$ yields a truncation rank $r_v$ and a reduced input lifting
\begin{equation}
    \tilde\Psi_u(u) \;:=\; U_{v,r_v}^\top\, \Psiu(u) \;\in\; \Real^{r_v},
\label{eq:reduce-v}
\end{equation}
with $\tilde v := U_{v,r_v}^\top v$. The tolerances $\tau_z$ and $\tau_v$ are selected by a coarse bisection on the multi-step feature-prediction score $r^z_N$ of~\eqref{eq:multistep-profile}, evaluated on the selected windows: among the candidate tolerances whose score lies within a small relative band ($5\%$) of the best, the one yielding the smallest rank is retained, trading a negligible loss in prediction accuracy for a lower-dimensional model and faster online optimization. \blue{Although $\tau_z, \tau_v$ and the horizon-selection tolerance $\epsilon$ of~\eqref{eq:Hp-rule} all trade accuracy against complexity, their roles differ: $\tau_z, \tau_v$ govern \emph{compression} and are selected automatically by the bisection above, whereas $\epsilon$ is a user-specified accuracy level that determines the MPC horizon for the already-reduced model. Note also that, while the state and input reductions are computed independently, $U_{z,r_z}$ and $U_{v,r_v}$ are \emph{static} linear maps applied sample-wise; the reduced snapshots therefore inherit exactly the time alignment of the raw data, and the pairing $(\tilde z_t, \tilde v_t, \tilde z_{t+1})$ used in the regression is unaffected.}

After the reduction step we obtain the following reduced order Khatri--Rao GeKo realization:
\begin{equation}
\begin{split}
    \tilde z_{t+1} \;&=\; \tilde K \big(\tilde z_t \kron \tilde v_t\big),\\
    \tilde K \;&=\; U_{z,r_z}^\top\, K \,
        \big(U_{z,r_z} \kron U_{v,r_v}\big) \in \Real^{r_z \times r_z r_v}.
        \end{split}
\label{eq:reduced-model}
\end{equation}
\ifl
The next result establishes that the resulting reduced model is itself a valid Khatri--Rao GeKo realization in the sense of~\eqref{eq:khatri-rao}, with a quantifiable approximation error.

\begin{theorem}
\label{thm:reduction}
Let $U_{z,r_z}$ and $U_{v,r_v}$ have orthonormal columns. Suppose the training data satisfy
\begin{equation}
    \|z_t - U_{z,r_z} U_{z,r_z}^\top z_t\| \le \sigma_{r_z+1}^z,
    \;\;
    \|v_t - U_{v,r_v} U_{v,r_v}^\top v_t\| \le \sigma_{r_v+1}^v
\label{eq:reduce-bound}
\end{equation}
for all $t$ in the training data. Then the reduced model \eqref{eq:reduced-model} is a valid Khatri--Rao GeKo realization, namely the Galerkin projection of $K$ onto the reduced product subspace $\spn(\tilde\Psi_x) \otimes \spn(\tilde\Psi_u)$, and satisfies
\begin{equation}
    \|U_{z,r_z}^\top z_{t+1} - \tilde z_{t+1}\|
    \;\le\; \|K\|_2 \cdot
        \big(\sigma_{r_z+1}^z + \sigma_{r_v+1}^v
           + \sigma_{r_z+1}^z \sigma_{r_v+1}^v\big).
\label{eq:reduce-error}
\end{equation}
\end{theorem}
\begin{IEEEproof}
With $z_t = U_{z,r_z} \tilde z_t + \delta_z$ and $v_t = U_{v,r_v} \tilde v_t + \delta_v$, where $\|\delta_z\| \le \sigma_{r_z+1}^z$ and $\|\delta_v\| \le \sigma_{r_v+1}^v$ by~\eqref{eq:reduce-bound}, the bilinearity of the Kronecker product gives
\[
    z_t \kron v_t
        = (U_{z,r_z} \kron U_{v,r_v})(\tilde z_t \kron \tilde v_t) + r_t,
\]
with the residual $r_t$ satisfying $\|r_t\| \le \sigma_{r_z+1}^z + \sigma_{r_v+1}^v + \sigma_{r_z+1}^z \sigma_{r_v+1}^v$. Substituting in $U_{z,r_z}^\top z_{t+1} = U_{z,r_z}^\top K (z_t \kron v_t)$ and using the definition of $\tilde K$ in~\eqref{eq:reduced-model} yields $\|U_{z,r_z}^\top z_{t+1} - \tilde z_{t+1}\| = \|U_{z,r_z}^\top K r_t\| \le \|K\|_2 \|r_t\|$, which is~\eqref{eq:reduce-error}. The Galerkin interpretation follows by observing that $U_{z,r_z}^\top$ is the orthogonal projection onto $\spn(\tilde\Psi_x)$ and $(U_{z,r_z} \kron U_{v,r_v})$ embeds $\spn(\tilde\Psi_x \kron \tilde\Psi_u)$ into $\spn(\Psix \kron \Psiu)$.
\end{IEEEproof}
\fi

\begin{remark}[Refit vs.\ projection]
\label{rem:refit}
In practice, $\tilde K$ is not computed via the projection formula in~\eqref{eq:reduced-model} but rather \emph{refit} by multi-step ridge regression in the reduced space, applying~\eqref{eq:multistep-soln} with the projected matrices $\tilde\Phimat_k := (U_{z,r_z}^\top \otimes U_{v,r_v}^\top) \Phimat_k$ and $\tilde W_k := U_{z,r_z}^\top W_k$. By orthonormality of $U_{z,r_z}, U_{v,r_v}$ the projection and refit estimates coincide in the noiseless infinite-data limit; in finite samples the refitted estimate is optimal in the reduced space and is therefore preferred.
\end{remark}
\begin{remark}[Kernel-agnostic]
\label{rem:kernel-agnostic}
The developed structured SVD reduction does not use the specific form of $\Psix$ or $\Psiu$. The two-stage SVD reduction therefore applies to any choice of state and input lifting dictionary -- random Fourier features, wavelet bases, neural-network encoders, or time-delay embeddings -- and preserves the Khatri--Rao GeKo realization in each case.
\end{remark}

\subsection*{Linear decoder}

The reduced lifted state $\tilde z$ is paired with a linear decoder $\hat x = D \tilde z$, where $D \in \Real^{\nx \times r_z}$ is identified by ridge regression on the training data,
\begin{equation}
    D \;=\; X \tilde Z^\top \big(\tilde Z \tilde Z^\top + \gamma I\big)^{-1},
\label{eq:decoder-soln}
\end{equation}
with $X = [x^{(\ell)}_t]$ and $\tilde Z = U_{z,r_z}^\top [\Psix(x^{(\ell)}_t)]$ collecting all training state--snapshot pairs. The corresponding lifted MPC weight is $Q_{\tilde z} := D^\top Q_x D$, so that the cost on $\tilde z$ penalizes only the linear combinations that decode to the physical state.

\section{Random Fourier Features Realization and Predictive Control Formulation}
\label{sec:rff}

The framework of Sections~\ref{sec:multistep}--\ref{sec:reduction} is agnostic to the choice of $\Psix$ and $\Psiu$. In this section we instantiate it with \emph{random Fourier features} (RFF), summarize the complete learning pipeline, and formulate the resulting predictive controller.

\subsection*{Random Fourier feature lifting}
We lift the state with the random Fourier feature map of~\cite{RahimiRecht2007},
which approximates a shift-invariant kernel by sampling frequencies from its
spectral density. For the Gaussian (RBF) kernel of length scale $\sigma_x$, the
frequencies are drawn $\omega_1,\ldots,\omega_{\nz} \stackrel{\rm iid}{\sim}
\mathcal{N}(0,\sigma_x^{-2}I)$ and the biases
$b_i \stackrel{\rm iid}{\sim} \mathrm{Unif}[0,2\pi]$, giving
\begin{equation}
    \Psix(x) \;=\; \sqrt{\tfrac{2}{\nz}}\,
        \big[\cos(\omega_1^\top x + b_1),\ldots,
             \cos(\omega_{\nz}^\top x + b_{\nz})\big]^\top.
\label{eq:rff-x}
\end{equation}
The input lifting $\Psiu(u) \in \Real^{\nv}$ is constructed analogously with
frequency samples $\omega_j' \sim \mathcal{N}(0,\sigma_u^{-2}I)$ and uniform
biases $b_j'$, for $j = 1,\ldots,\nv$.
\ifl
Three properties make this choice attractive for data-driven GeKo MPC. First, by~\cite{RahimiRecht2007}, $\sup_{x,y \in \X} | \Psix(x)^\top \Psix(y) - k(x,y) |$ converges to zero in probability as $\nz \to \infty$, so $\spn(\Psix)$ becomes dense in the reproducing kernel Hilbert space of $k$; combined with the Galerkin convergence of EDMD on dense dictionaries~\cite{KordaMezic2018,Mezic2021}, the empirical Khatri--Rao GeKo converges to the true generalized Koopman operator in the joint limit $(\nz, \nv, M_{\rm aug}) \to \infty$. Lifted linear and bilinear realizations do not share this density property for general input nonlinearities. Second, the normalization $\sqrt{2/\nz}$ in~\eqref{eq:rff-x} ensures $\|\Psix(x)\|^2 \approx 1$ uniformly in $x$ and $\nz$, so the conditioning of the regression~\eqref{eq:multistep-soln} does not deteriorate as $\nz$ grows and a single ridge parameter $\gamma$ suffices across a wide range of dictionary sizes. Third, each feature is a single trigonometric function, so the Jacobian has the closed form
\begin{equation}
    \partial_x\,\Psix(x)_i = -\sqrt{\tfrac{2}{\nz}}\, \omega_i\,
        \sin(\omega_i^\top x + b_i),
\label{eq:rff-jac}
\end{equation}
which can be exploited by analytic-gradient MPC solvers. 
\fi

The frequency spreads $\sigma_x, \sigma_u$ play the role of RBF kernel length scales and are set on the order of $\mathrm{range}(x)$ and $\mathrm{range}(u)$ respectively; the dictionary sizes $\nz, \nv$ are intentionally oversized, since the SVD reduction of Section~\ref{sec:reduction} compresses them to the effective rank of the data manifold.

\subsection*{Predictive control formulation}

Letting $\Delta u_k:=u_k-u_{k-1}$, at each time instant $k$, MPC solves
\begin{equation}
\begin{aligned}
    \min_{u_{0:N_\text{mpc}-1}}\;\;
        & \sum_{k=0}^{N_\text{mpc}}
            \|D \tilde z_k - x_{\rm ref}\|^2_{Q_x}
            +\sum_{k=0}^{N_\text{mpc}-1}\!\!
            (\!\|u_k\|_{R_u}^2
                  + \|\Delta u_k\|_{R_{\Delta u}}^2\!) \\
    \text{s.t.}\quad
        & \tilde z_{k+1} = \tilde K\big(\tilde z_k \kron \tilde v_k\big),
          \;\; \tilde v_k = \tilde\Psi_u(u_k), \\
        & \tilde z_0 = U_{z,r_z}^\top \Psix(x_t),
          \;\; u_k \in \U, \;\; k = 0,\ldots,N_\text{mpc}-1.
\end{aligned}
\label{eq:mpc}
\end{equation}
The first input $u_0^\star$ of the optimal sequence is applied to the plant, the state is measured at the next sampling instant, and the optimization is repeated. The input-rate penalty promotes smooth control and dampens chattering. 
\ifl
\begin{remark}[Soft state constraints]
\label{rem:soft-state}
Polytopic state constraints $F x \le f$ map, through the linear decoder, to linear constraints $F D \tilde z \le f$ on the lifted variable. Strict enforcement may be infeasible due to input saturation reachability limits or residual model error; softening with non-negative slacks $\xi_k \in \Real^{n_c}$ and an exact-penalty term $\sum_{k=1}^{N} (\rho_1^\top \xi_k + \xi_k^\top \diag(\rho_2)\xi_k)$ subject to $F D \tilde z_k - f \le \xi_k, \xi_k \ge 0$, preserves recursive feasibility while penalizing violations measurably; for sufficiently large $\rho_1$ the slacks remain zero whenever the hard problem is feasible~\cite{KerriganMaciejowski2000}.
\end{remark}
\fi

Algorithm~\ref{alg:konverge} summarizes the GeKo MPC pipeline driven by data from end-to-end.

\begin{algorithm}[t]
\caption{Data-driven Generalized Koopman MPC}
\label{alg:konverge}
\begin{algorithmic}[1]
\STATE Generate $L$ training trajectories with the multi-frequency excitation~\eqref{eq:input-design}.
\STATE Sample RFF parameters $\{(\omega_i,b_i)\}_{i=1}^{\nz}$ and $\{(\omega_j',b_j')\}_{j=1}^{\nv}$ as in~\eqref{eq:rff-x}.
\STATE Lift the trajectories to $\{z_t^{(\ell)}\}, \{v_t^{(\ell)}\}$ and compute the SVD reductions $U_{z,r_z}, U_{v,r_v}$ per~\eqref{eq:reduce-z}--\eqref{eq:reduce-v}, with the truncation tolerances chosen by bisection on the multi-step score $r^x_N$ of~\eqref{eq:multistep-profile}.
\STATE Project the lifted trajectories, $\tilde z_t = U_{z,r_z}^\top z_t$, $\tilde v_t = U_{v,r_v}^\top v_t$.
\STATE Collect time-shifted snapshot windows~\eqref{eq:window} and select a representative subset $\mathcal{S}$ of size $M_a$ by $k$-means clustering of windows of initial states.
\STATE Build the augmented matrices $\tilde\Phimat, \tilde W$ as in~\eqref{eq:aug-data} (in the reduced space, cf.\ Remark~\ref{rem:refit}) and solve the multi-step ridge regression~\eqref{eq:multistep-soln} for $\tilde K$.
\STATE Fit the linear decoder $D$ per~\eqref{eq:decoder-soln} and form $Q_{\tilde z} = D^\top Q_x D$.
\STATE Compute the multi-step profile $\{r^x_k\}_{k=1}^{N_d}$ and choose the MPC horizon $N_\text{mpc}$ via~\eqref{eq:Hp-rule}.
\STATE Deploy the predictive controller~\eqref{eq:mpc}.
\end{algorithmic}
\end{algorithm}


\section{Validation on the chaotic Lorenz system}
\label{sec:experiments} 
We consider the controlled Lorenz system
\begin{equation}
\begin{aligned}
    \dot x_1 &= \sigma\,(x_2 - x_1), \\
    \dot x_2 &= x_1(\rho - x_3) - x_2 + u, \\
    \dot x_3 &= x_1 x_2 - \beta\,x_3,
\end{aligned}
\label{eq:lorenz}
\end{equation}
with the canonical chaotic parameters $\sigma = 10$, $\rho = 28$, $\beta = 8/3$, and a single scalar input $u$ entering the second state equation.\footnote{Actuation on $\dot x_2$ follows the chaos-control formulation of~\cite{Yang2002control}.} The uncontrolled flow possesses the two symmetric unstable equilibria $\pm\,x^\star$ with $x^\star = \big(\sqrt{\beta(\rho-1)},\,\sqrt{\beta(\rho-1)},\,\rho-1\big)^\top \approx (8.49,\,8.49,\,27)^\top$, around which a typical trajectory orbits chaotically. The control objective is to stabilize the equilibrium $x^\star$, i.e.\ $x_{\rm ref} = x^\star$ in~\eqref{eq:mpc}. The continuous dynamics are sampled at $T_s = 0.01$, giving the discrete-time plant on which the GeKo model and the predictive controller~\eqref{eq:mpc} operate.
 
We generate $L = 8$ training trajectories of $5000$ samples each ($50$~s of data, $4\times10^4$ state--input pairs total). Each trajectory is initialized from a point drawn uniformly from the box $x_{1},x_{2}\in[-20,20]$, $x_{3}\in[-5,55]$ that covers the attractor, and is forced by the multi-frequency excitation~\eqref{eq:input-design}
\begin{equation}
    u^{(\ell)}_k \;=\; \mathrm{sat}_{[-30,30]}\!\Big(
        \textstyle\sum_{m=1}^{6} a\,\sin(\omega_m k T_s + \phi^{(\ell)}_m)\Big),
\label{eq:lorenz-input}
\end{equation}
with amplitude $a = 6$, fixed angular frequencies $\omega_m \in \{0.3,0.7,1.1,1.9,3.1,5.3\}$~rad/s, and phases $\phi^{(\ell)}_m \sim \mathrm{Unif}[0,2\pi]$ resampled per trajectory; the excitation is saturated to the deployment bound $\U = [-30,30]$, so that the training data exercise the full admissible input range. Trajectories are integrated with an adaptive Runge--Kutta scheme (\texttt{ode45}, relative and absolute tolerances $10^{-8}, 10^{-10}$); inside the MPC loop the plant is advanced with a fixed-step RK4 using four substeps per $T_s$, which reproduces the adaptive solution within $2\times10^{-8}$ accuracy on the attractor at a fraction of the cost.
 
The state is lifted with $\nz = 400$ random Fourier features~\eqref{eq:rff-x} of length scale $\sigma_x = 10$, and the scalar input with $\nv = 20$ features of length scale $\sigma_u = 7.5$. As the data manifold is only 3D, the lifted state features are highly redundant; the SVD reduction~\eqref{eq:reduce-z} with the bisection tolerance rule of Section~\ref{sec:reduction} compresses $\nz = 400 \to r_z = 150$ while keeping the multi-step score~\eqref{eq:multistep-profile} within $5\%$ of its minimum, and the input reduction~\eqref{eq:reduce-v} compresses $\nv = 20 \to r_v = 13$. All subsequent regression and control are carried out in the reduced coordinates $\tilde z = U_{z,r_z}^\top \Psix(x)$, $\tilde v = U_{v,r_v}^\top \Psiu(u)$.
 
We form time-shifted windows of length $N_d = 20$~\eqref{eq:window}\blue{, yielding $39{,}840$ raw windows of which $12{,}000$ are kept by uniform subsampling. We then} retain a representative subset of \blue{$M_a = 1500$} windows by $k$-means clustering on the concatenated anchor snapshots $[\tilde z_0;\,\tilde z_{N_d/2};\,\tilde z_{N_d}]$ (window-aware selection, so that windows sharing a start point but diverging on the attractor are kept distinct), and solve the multi-step ridge regression~\eqref{eq:multistep-soln} \blue{over $M_{\rm aug} = N_d M_a = 30{,}000$ stacked snapshot pairs} with $\gamma = 10^{-4}$ for the reduced Khatri--Rao operator $\tilde K \in \Real^{r_z \times r_z r_v}$ ($150 \times 1950$). The linear decoder $D$ is fit by~\eqref{eq:decoder-soln} and induces the lifted state weight $Q_{\tilde z} = D^\top Q_x D$ used in~\eqref{eq:mpc}.

\blue{In terms of offline computational cost, the pipeline is dominated by the one-off reduction and selection steps: on the same machine used for the closed-loop study below, the SVD and tolerance bisection took $139$~s and the window-aware $k$-means selection $101$~s, while solving the multi-step ridge regression~\eqref{eq:multistep-soln} itself took only $6.5$~s at $M_{\rm aug}=30{,}000$ ($0.34$~s for the $N_d{=}1$ special case). The closed-form multi-step estimator thus scales mildly with $M_{\rm aug}$ and is not the bottleneck of the offline phase.}
 
\begin{figure}[t]
  \centering
  \includegraphics[width=0.8\linewidth]{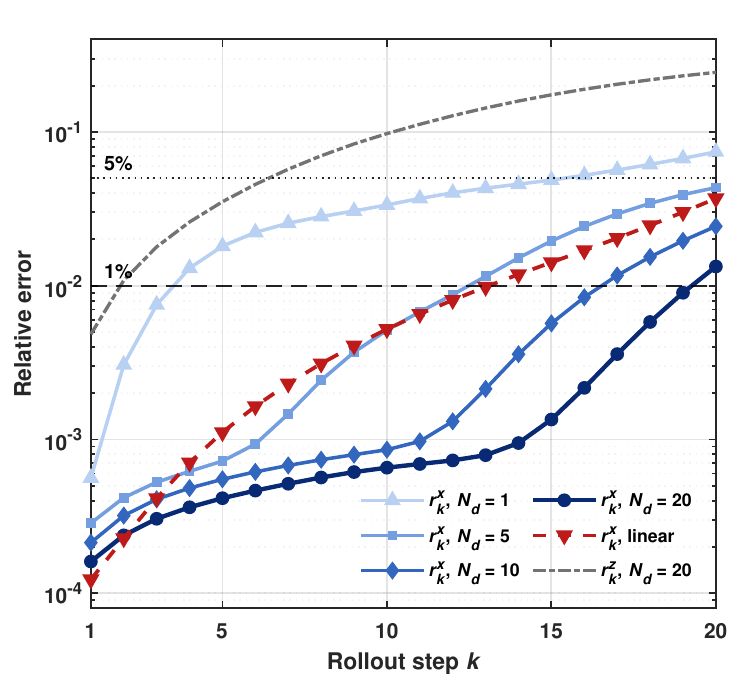}
  \caption{\blue{Multi-step open-loop state prediction error $r^x_k$~\eqref{eq:multistep-profile}: GeKo with $N_d \in \{1,5,10,20\}$ (blue, light to dark), linear Koopman (red), and $r^z_k$ for GeKo with $N_d{=}20$ (gray).}
  }
  \label{fig:lorenz-rrmse}
\end{figure}
\blue{Figure~\ref{fig:lorenz-rrmse} reports the multi-step state error profiles $\{r^x_k\}$ of~\eqref{eq:multistep-profile} for GeKo models fitted with $N_d \in \{1,5,10,20\}$ on the same selected windows and the same reduced dictionary, where $N_d = 1$ recovers the one-step EDMD-type estimator~\eqref{eq:edmd}, together with a lifted linear Koopman baseline $\tilde z_{t+1} = A \tilde z_t + B u_t$ fitted on the same $30{,}000$ pairs. The profiles improve monotonically with $N_d$: at $k = 20$ the relative error decreases from $7.4\%$ ($N_d{=}1$) to $1.3\%$ ($N_d{=}20$), a $5.6\times$ reduction attributable to the multi-step regression alone. A sample-count-matched one-step baseline, fitted on $30{,}000$ one-step pairs drawn randomly from the same trajectories, still attains $7.5\%$ at $k=20$, confirming that the improvement stems from exposing the operator to evolved snapshots along trajectory windows rather than from the larger sample count. The linear baseline is accurate at short horizons, consistent with control-affine dynamics~\eqref{eq:lorenz}, but degrades to $3.7\%$ at $k = 20$, i.e., $2.8\times$ above the bilinear GeKo model.}
The multi-step state profile $\{r^x_k\}$ \blue{of the final $N_d{=}20$ model} stays below $1\%$ relative error out to $k = 19$ steps and below $5\%$ through $k = 20$, so the horizon rule~\eqref{eq:Hp-rule} admits any $N_\text{mpc} \le 20$; we use $N_\text{mpc}=12$ in the closed-loop study below\blue{, which balances prediction accuracy against per-step computation: the SQP solve time grows with the horizon, and $N_\text{mpc}=12$ already achieves fast stabilization, as shown next}. The MPC controller solves~\eqref{eq:mpc} with state weight $Q_x = I_3$, input weight $R_u = 10^{-2}$, a small input-rate penalty $R_{\Delta u} = 10^{-3}$, and the box constraint $\U = [-30,30]$. The lifted reference is $\tilde z_{\rm ref} = U_{z,r_z}^\top \Psix(x^\star)$, recomputed once offline. At each step the reduced Koopman bilinear model $\tilde z_{k+1} = \tilde K(\tilde z_k \kron \tilde v_k)$ with $\tilde v_k = U_{v,r_v}^\top \Psiu(u_k)$ is used as the prediction model. The resulting nonlinear program is solved with MATLAB's \texttt{fmincon} (SQP algorithm, optimality tolerance $10^{-4}$, at most $100$ iterations per step), warm-started at each step with the shifted previous solution $[\,u_{1:N-1}^\star;\,0\,]$.

\begin{figure}[t]
  \centering
  \includegraphics[width=0.8\linewidth]{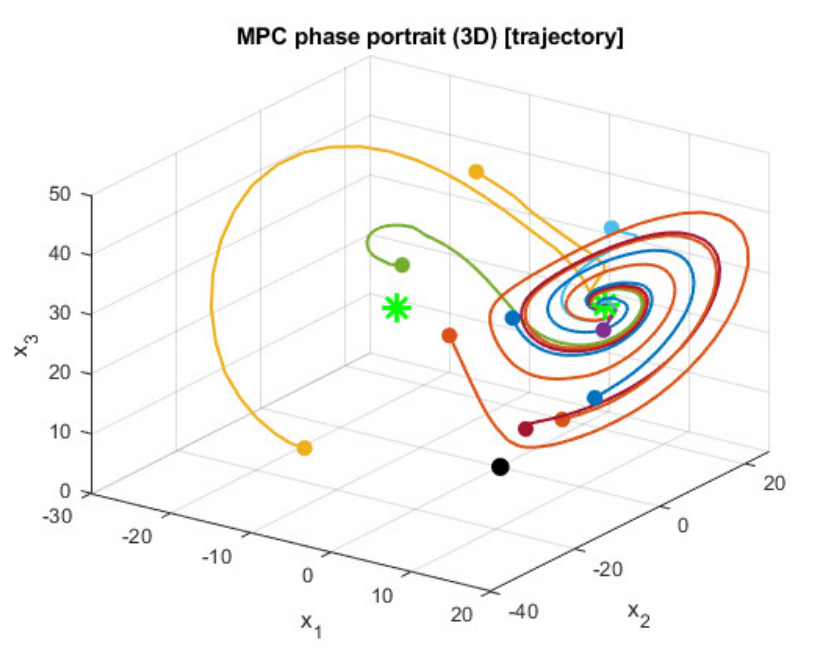}
  \caption{Phase-space view of closed-loop trajectories corresponding to $10$ initial conditions.
    }
  \label{fig:lorenz-phase}
\end{figure}
\begin{figure}[t]
  \centering
  \includegraphics[width=0.68\linewidth]{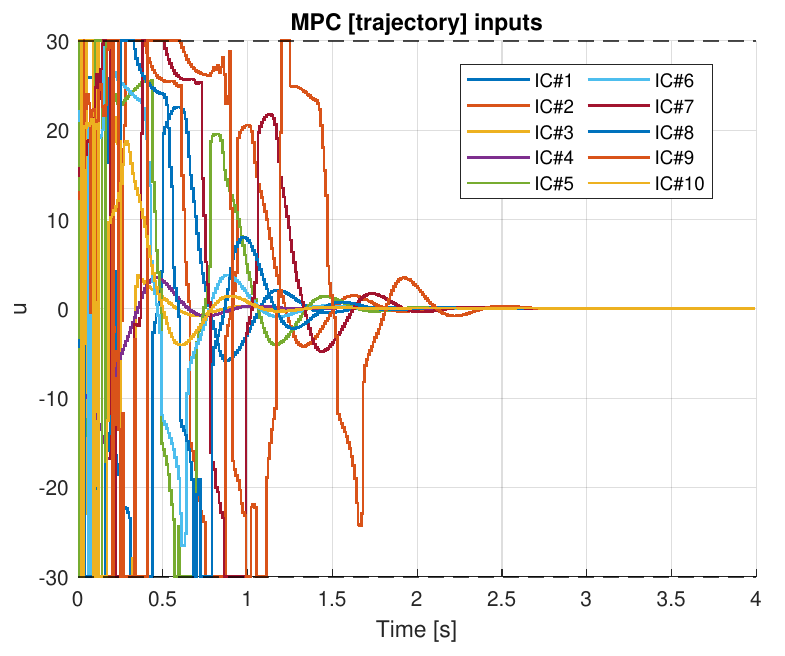}
  \caption{MPC control inputs for the $10$ closed-loop runs.}
  \label{fig:lorenz-inputs}
\end{figure}

We assess closed-loop performance from $10$ initial conditions (ICs): the
nominal $x_0 = (1,1,25)^\top$ together with $9$ reproducible draws
(fixed seed) from the training box $x_{1},x_{2}\in[-20,20]$,
$x_{3}\in[-5,55]$. The resulting closed-loop trajectories in state space and the input signals are reported in Figures~\ref{fig:lorenz-phase} and~\ref{fig:lorenz-inputs}, respectively. Averaged over the ten initial conditions, the controller drives the state to a final error of $\|x(T)-x^\star\| = 8\times10^{-3}$, with a mean per-step solve time of $101$~ms (maximum $147$~ms across all runs) and an input that respects the bound $|u|\le 30$ throughout. 

Every run stabilizes the common target $x^\star$ in less than $2$ s, despite starting from both attractor lobes and from off-attractor states, which demonstrates the effectiveness of the developed data-driven Koopman learning framework for MPC. \blue{For comparison, we also ran the closed-loop study with the lifted linear baseline as the prediction model, in which case the MPC problem becomes a convex QP solved at a fraction of the computational cost: closed-loop performance was comparable. This outcome is in fact predicted by the horizon rule~\eqref{eq:Hp-rule}, as the linear profile in Figure~\ref{fig:lorenz-rrmse} satisfies $r^x_{12} \approx 0.8\% < \epsilon$ at the deployed horizon, while receding-horizon re-planning compensates its moderate long-horizon degradation. For a DC-motor example where the linear Koopman MPC fails structurally we refer to \cite{MezicCDCTutorial2026} due to space limits.}

In terms of computational cost, the per-step optimization averaged $101$~ms which is not yet real-time compatible, but a generic MATLAB SQP solver was used on a nonconvex program. \blue{A tailored solver exploiting the bilinear Khatri--Rao structure of the prediction model is a subject of future work.}
 
\section{Conclusions}
\label{sec:conclusions}
 
In this paper we presented a data-driven framework for learning finite-dimensional approximate generalized Koopman operator models for predictive control. Firstly, we recalled the specific form of the generalized Koopman operator for control, which yields a Khatri--Rao regression problem. Then we developed a time-sequenced multi-step method to reduce the multi-step prediction error of the model fitted via kernel regression. To reduce the dimension of the generalized Koopman model we used a structured SVD approach which eliminates redundant features. Finally, the complete framework was implemented in an MPC pipeline and validated on the chaotic Lorenz system, using random Fourier features. 

Future work will deal with further optimizing GeKo models for MPC and tailored efficient solvers \blue{exploiting the bilinear Khatri--Rao structure, with robustness to measurement and process noise, and with validation on multi-input systems with strongly non-affine actuation.}

\bibliographystyle{IEEEtran}
\bibliography{references}

\end{document}